\documentclass{amsart}
\usepackage{mathrsfs}
\usepackage[all]{xy}

\usepackage{wasysym}
\usepackage{amssymb}

\usepackage{ifpdf}
\ifpdf 
  \usepackage[pdftex]{graphicx}
  \DeclareGraphicsExtensions{.pdf,.png,.jpg,.jpeg,.mps}
  \usepackage{pgf}
\else 
  \usepackage{graphicx}
  \DeclareGraphicsExtensions{.eps,.bmp}
  \usepackage{pgf}
  \usepackage{pstricks}
\fi
\usepackage{epic,bez123}
\usepackage{wrapfig}
\usepackage{comment}

\newtheorem{thm}{Theorem}[section]
\newtheorem{prop}[thm]{Proposition}
\newtheorem{cor}[thm]{Corollary}
\newtheorem{lem}[thm]{Lemma}
\newtheorem{lema}{Lemma}
\newtheorem{cora}{Corollary}

\newtheorem*{claim}{Claim}

\newtheorem*{conv}{Convention}

\theoremstyle{definition}
\newtheorem{defn}[thm]{Definition}

\theoremstyle{remark}
\newtheorem{quest}[thm]{Question}
\newtheorem*{rem}{Remark}
\newtheorem{example}[thm]{Example}

\newtheorem*{ack}{Acknowledgment}

\newcommand{\Gx }{\mathscr{G} (G, S)}
\newcommand{\Gy }{\mathscr{G} (G, T)}

\newcommand{\act}{\curvearrowright}

\newcommand{\g}[1]{\delta_{#1}}

\newcommand{\Gf}{\overline{G}_{S, f}}
\newcommand{\pGf}{\partial_{S, f}{G}}

\newcommand{\PS}[1]{\Theta_{#1}(s)}

\newcommand{\diam }[1]{{\|#1\|}}

\newcommand{\proj }{{{\Pi}}}

\newcommand{\len }{\ell}

\usepackage[bookmarks=true, pdfauthor={YANG Wenyuan}]{hyperref}

\begin{document}

\title{Growth tightness for groups with contracting elements}

\author{Wen-yuan Yang}

\address{Beijing International Center for Mathematical Research \&
School of Mathematical Sciences, Peking University, Beijing, 100871,
P.R.China}

\address{Le Département de Mathématiques de la Faculté des
Sciences d'Orsay, Université Paris-Sud 11, France}

\email{yabziz@gmail.com}
\thanks{This research is supported by the ERC starting grant GA 257110 "RaWG"}



\date{Dec 14, 2013}

\dedicatory{}


\begin{abstract}
We establish growth tightness for a class of groups acting
geometrically on a geodesic metric space and containing a
contracting element. As a consequence,  any group with nontrivial
Floyd boundary are proven to be growth tight with respect to word
metrics. In particular, all non-elementary relatively hyperbolic
group are growth tight. This generalizes previous works of
Arzhantseva-Lysenok and Sambusetti. Another interesting consequence
is that CAT(0) groups with rank-1 elements are growth tight with
respect to CAT(0)-metric.
\end{abstract}

\maketitle

\section{Introduction}
In \cite{GriH}, R. Grigorchuk and P. de la Harpe introduced the
notion of growth tightness for a finitely generated group. Let $G$
be a group with a finite generating set $S$. Assume that $1 \notin
S$ and $S = S^{-1}$. Consider the word metric $d_S$ on $G$. Denote
$B(n)=\{g \in G: d_S(1, g) \le n\}$ for $n \ge 0$. The
\textit{growth rate} $\delta_{G, S}$ of $G$ relative to $S$ is the
limit
$$
\delta_{G, S} = \lim\limits_{n \to \infty} n^{-1} \ln \sharp B(n),
$$
which exists by the inequality that $\sharp B(n+m) \le \sharp B(n)
\cdot \sharp B(m)$.

For any quotient $\bar G$ of $G$, the $S$ naturally induces a finite
generating set $\bar S$ of $\bar G$. Clearly, $\g {G, S} \ge \g
{\bar G, \bar S}$. Thus, the property of growth tightness is
requiring the inequality to be strict. More precisely,
\begin{defn}
The pair $(G, S)$ is called \textit{growth tight} if $\delta_{G, S}
> \delta_{\bar G, \bar S}$ for any quotient $\bar G=G/\Gamma$,
where $\Gamma$ is an infinite normal subgroup, and $\bar S$ is the
canonical image of $S$ in $\bar G$. We say that $G$ is
\textit{growth tight} if it is growth tight with respect to every
finite generating set.
\end{defn}
\begin{rem}
For notational simplicity, we omit the index $S, \bar S$ in $\g {G,
S}, \g {\bar G, \bar S}$ if it is understood. We exclude that
$\Gamma$ is finite, as we always have $\g G=\g {\bar G}$ in this
case.
\end{rem}

The concept of growth tightness turns out to reflect some negative
curvature features of groups. Recall that an infinite group is
called \textit{elementary} if it is a finite extension of a cyclic
group. In \cite{AL},  Arzhantseva-Lysenok proved that all
non-elementary hyperbolic groups are growth tight. Later, Sambusetti
proved in \cite{Sam2} that the free product of two groups is growth
tight, unless it is elementary. On the other hand, the direct
product of any two groups is not growth tight with respect to some
generating set.

In \cite{Sam1}, Sambusetti also proposed a study of generalization
of growth tightness for any proper left-invariant (pseudo-)metric
$d$ on $G$. Recall that $d$ is \textit{proper} if any closed ball of
finite radius is finite, and is \textit{left-invariant} if $d(h_1,
h_2)=d(gh_1, gh_2)$ for any $g, h_1, h_2 \in G$.  Such pseudo-metric
$d$ on $G$ naturally arises as the pullback of a metric $d$ on a
metric space $Y$ on which $G$ acts properly. In other words, we are
actually talking about growth tightness for a proper action.

Let $G$ be a group acting properly on a geodesic metric space $(Y,
d)$. Fix a basepoint $o \in Y$. Denote $B(o, n) = \{go: g \in G,
d(o, go) \le n\}$ for $n\in \mathbb N$.  The \textit{growth rate}
$\g A$ of a subset $A \subset G$ relative to $d$ is thus defined as
$$
\g A = \limsup\limits_{n \to \infty} n^{-1} \ln \sharp (B(o, n) \cap
Ao).
$$
Note that $\g A$ does not depend on the choice of $o$.

Let $\Gamma$ be a normal subgroup in $G$. Then $\bar G = G/\Gamma$
acts on $\bar X=X/\Gamma$ by $\Gamma g \cdot (\Gamma x) \to \Gamma
gx$. Equip $\bar X$ with the quotient metric $\bar d$. We call $\bar
G$ a \textit{proper} quotient if $\Gamma$ is infinite. Denote by
$\g{\bar G}$ the growth rate of $\bar G$ relative to $\bar d$.
\begin{defn}
The action $G \act (X, d)$ is called \textit{growth tight} if $\g{G}
> \g{\bar G}$ for any proper quotient $\bar G$ of $G$.
We can also say that $G$ is growth tight with respect to $d$.
\end{defn}

In recent years growth tightness for groups with respect to various
classes of negative curvature metrics are established. In
\cite{Sam3}, growth tightness is proved for a group acting properly
and cocompactly on a simply connected Riemannian manifold with
negative curvature. In \cite{DPPS}, Dalbo-Peigne-Picaud-Sambusetti
generalized this result under some natural conditions, for a
discrete group on complete simply connected Riemannian manifolds
with pinched negative curvature. We refer the reader to \cite{AL},
\cite{Sam2}, \cite{Sam3}, \cite{Sab}, \cite{DPPS} for further
references and applications therein.

The present paper is a continuation of above studies and aiming to
show the growth tightness for a class of \textit{groups with a
contracting element}. Fix a preferred class of quasigeodesics
$\mathcal L$, which contains at least all geodesics in $Y$.
Informally speaking, a subset $X$ has \textit{contracting} property
if any path in $\mathcal L$ far from $X$ has an uniform bounded
projection to $X$. An element $h$ in $G$ is called
\textit{contracting}, if some (or any) orbit of $\langle h\rangle$
in $Y$ is contracting.

Recall that an action $G \act (Y, d)$ is \textit{geometric} if it is
proper and cocompact. Our main result can be thus formulated as
follows.
\begin{thm}\label{tight}
If a non-elementary group $G$ admits a geometric action on a metric
space $(Y, d)$ with a contracting element, then $G$ is growth tight
with respect to $d$.
\end{thm}

We now describe several examples of groups with a contracting
element. First of all, it is well-known that a hyperbolic isometry
on a hyperbolic space is contracting. Thus as a corollary, we
recover a result of Sabourau in \cite{Sab} as follows.
\begin{cor}\label{tightHYP}
Suppose $G$ admits a geometric action on a geodesic hyperbolic space
$(Y, d)$. Then $G$ is growth tight with respect to the metric $d$.
\end{cor}

The next interesting examples are CAT(0) groups with a rank-1
element. Namely, these are groups acting geometrically on a CAT(0)
space with a \textit{rank-1} element, which is a hyperbolic element
such that every axe does not bound a half plane. In \cite{BF}, it is
proved that every axe of a rank-1 element is contracting. Hence, we
obtain the following result.

\begin{cor}\label{tightCAT}
Suppose a non-elementary group $G$ admits a geometric action on a
CAT(0) space $(Y, d)$. If $G$ contains a rank-1 element, then $G$ is
growth tight with respect to the metric $d$.
\end{cor}

The final, and most important, examples in this paper are groups
with nontrivial Floyd boundary.  The notion of Floyd boundary was
introduced in \cite{Floyd} to compactify a finitely generated group.
Floyd boundary $\pGf$ depends on the choice of a finite generating
set $S$ and a \textit{Floyd function} $f$, cf. Section 7 for
definitions. Observe that the non-triviality of Floyd boundary is a
property of the pair $(G, S)$ in a certain sense (cf. Lemma
\ref{nontri}). We further note that every hyperbolic element gives a
contracting element for the action of $G$ on its Cayley graph. Thus,
by Theorem \ref{tight}, we obtain our second main result.

\begin{thm}\label{tightfloyd}
If $G$ admits a non-trivial Floyd boundary $\pGf$ for some finite
generating set $S$ and Floyd function $f$, then $G$ is growth tight
with respect to every finite generating set.
\end{thm}

In \cite{Ge2}, Gerasimov proved that all non-elementary relatively
hyperbolic groups have non-trivial Floyd boundary. As a consequence,
we obtain the following corollary, extending the results of
Arzhantseva-Lysenok \cite{AL} and Sambusetti \cite{Sam2}.
\begin{cor}\label{tighthyp}
Non-elementary relatively hyperbolic groups are growth tight.
\end{cor}

Let's remark about a relation between relatively hyperbolic groups
and groups with non-trivial Floyd boundary. As aforementioned, every
non-elementary relatively hyperbolic groups have non-trivial Floyd
boundary. It was asked by Olshanskii-Osin-Sapir in \cite{OOS}
whether its converse is true. They formulated the question for the
type of Floyd function $f(n)=1/(n^2 +1)$.

In \cite{YANG3}, we showed that their question is equivalent to the
question concerning about peripheral structures: whether every group
hyperbolic relative to a collection of \textit{Non-Relatively
Hyperbolic} subgroups acts geometrically finitely on its Floyd
boundary. Here a NRH group refers to a group not hyperbolic relative
to any collection of proper subgroups. Therefore, it is yet unknown
that whether the class of groups with non-trivial Floyd boundary
coincide with the class of (non-elementary) relatively hyperbolic
groups.

Interestingly, we prove in the appendix that every group with
non-trivial Floyd boundary contains a hyperbolically embedded
subgroup. The class of groups with hyperbolically embedded subgroups
are proposed in \cite{DGO} by Dahmani-Guirardel-Osin as a natural
generalization of relatively hyperbolic groups. It is worth to point
out that there are many groups with hyperbolically embedded
subgroups but which have trivial Floyd boundary. For instance,
mapping class groups with a few exceptions have trivial Floyd
boundary, see \cite{KaNos} for more examples.

The structure of this paper is as follows. In section 2, we discuss
a contracting property of a subset and deduce some preliminary
results.  Section 3 is devoted to a key notion of an
\textit{admissible path} in a geodesic metric space with a system of
contracting subsets. Our technical result is that a 'long'
admissible path is a quasigeodesic. Section 4 brings in a group
action over the metric space and defines a group with a contracting
element. We show that any infinite normal subgroup contains
infinitely many contracting elements. In Section 5, we briefly
recall a critical gap criterion of Dalbo-Peigne-Picaud-Sambusetti
\cite{DPPS}.  In Section 6,  these ingredients are brought together
to prove Theorem \ref{tight}. We describe in some detail in Section
7 about groups with non-trivial Floyd boundary and explain a proof
of Theorem \ref{tightfloyd}. Some open problems are posed in Section
8. In the appendix, we prove that groups with a contracting element
contains a hyperbolically embedded subgroup.

\section{Contracting sets}

The purpose of this section is to discuss a notion of a contracting
subset in a geodesic metric space. The motivating examples are
parabolic subgroups in relatively hyperbolic groups and contracting
segments in CAT(0) spaces.

\subsection{Notations and Conventions}
Let $(Y, d)$ be a geodesic metric space. Given a subset $X$ and a
number $r \ge 0$, let $N_r(X) = \{y \in Y: d(y, X) \le r \}$ be the
closed neighborhood of $X$ with radius $r$. Denote by $\diam {X}$
the diameter of $X$ with respect to $d$.

Fix a (sufficiently small) number $\delta > 0$ that won't change in
the rest of paper. Given a point $y \in Y$ and subset $X \subset Y$,
let $\proj_X (y)$ be the set of points $x$ in $X$ such that $d(y, x)
\le d(y, X) + \delta$. Define \textit{the projection} of a subset
$A$ to $X$ as $\proj_X(A) = \cup_{a \in A} \proj_X(a)$.

Let $p$ be a path in $Y$ with initial and terminal endpoints  $p_-$
and $p_+$ respectively. Denote by $\len (p)$ the length of $p$.
Given two points $x, y \in p$, denote by $[x,y]_p$ the subpath of
$p$ going from $x$ to $y$.

Let $p, q$ be two paths in $Y$. Denote by $p\cdot q$(or $p q$ if it
is clear in context) the concatenated path provided that $p_+ =
q_-$.

A path $p$ going from $p_-$ to $p_+$ induces a first-last order as
we describe now. Given a property (P), a point $z$ on $p$ is called
the \textit{first point} satisfying (P) if $z$ is among the points
$w$ on $p$ with the property (P) such that $\len([p_-, w]_p)$ is
minimal. The \textit{last point} satisfying (P) is defined in a
similarly way.

Let $f(x,y): \mathbb R \times \mathbb R \to \mathbb R_+$ be a
function. For notational simplicity, we frequently write $f_{x,y} =
f(x,y)$.

\subsection{Contracting subsets}

\begin{defn}[Contracting subset]
Suppose $\mathcal L$ is a preferred collection of quasigeodesics in
$X$. Let $\mu: \mathbb R \times \mathbb R \to \mathbb R_+$ and
$\epsilon: \mathbb R \times \mathbb R \to \mathbb R_+$ be two
functions.

Given a subset $X$ in $Y$, if the following inequality holds
$$\diam{\proj_{X} (q)} < \epsilon(\lambda, c),$$
for any $(\lambda, c)$-quasigeodesic $q \in \mathcal L$ with $d(q,
X) \ge \mu(\lambda, c)$, then $X$ is called $(\mu,
\epsilon)$-\textit{contracting} with respect to $\mathcal L$. A
collection of $(\mu, \epsilon)$-contracting subsets is referred to
as a $(\mu, \epsilon)$-\textit{contracting system} (with respect to
$\mathcal L$).

\end{defn}

\begin{rem}
The terminology "contracting segments" was used by
Bestivina-Fujiwara in \cite{BF2}. It is easy to see that the
contracting notion in their sense is equivalent to ours, cf.
Corollary 3.4 in \cite{BF2}.
\end{rem}

\begin{example} \label{examples} We note the following examples in various contexts.
\begin{enumerate}
\item
Quasigeodesics and quasiconvex subsets are contracting with respect
to the set of all quasigeodesics in hyperbolic spaces. These are
best-known examples in the literature.
\item
Fully quasiconvex subgroups (and in particular, maximal parabolic
subgroups) are contracting with respect to the set of all
quasigeodesics in relatively hyperbolic groups (see Proposition
8.2.4 in \cite{GePo4}).
\item
The subgroup generated by a hyperbolic element is contracting  with
respect to the set of all quasigeodesics in groups with non-trivial
Floyd boundary.  This will be described in Section 7.
\item
Contracting segments in CAT(0)-spaces in the sense of in
Bestvina-Fujiwara are contracting here with respect to the set of
geodesics (see Corollary 3.4 in \cite{BF2}).
\item
Any finite neighborhood of a contracting subset is still contracting
with respect to the same $\mathcal L$.
\end{enumerate}
\end{example}

\begin{conv}
In view of Examples \ref{examples}, the preferred collection
$\mathcal L$ in the sequel is always assumed to be containing all
geodesics in $Y$.
\end{conv}

\begin{defn}[Quasiconvexity]
Let $\sigma: \mathbb R \to \mathbb R_+$ be a function. A subset $X
\subset Y$ is called \textit{$\sigma$-quasiconvex} if given $U \ge
0$,  any geodesic with endpoints in $N_U(X)$ lies in the
neighborhood $N_{\sigma(U)}(X)$.
\end{defn}

Clearly, the contracting property implies the quasiconvexity. The
proof is straightforward and left to the interested reader.

\begin{lem}\label{quasiconvexity}
Let $X$ be a $(\mu, \epsilon)$-contracting subset in $Y$. Then there
exists a function $\sigma: \mathbb R \to \mathbb R_+$ such that $X$
is $\sigma$-quasiconvex.
\end{lem}


We now shall introduce an additional property, named bounded
intersection property for a contracting system.
\begin{defn}[Bounded Intersection]
Given a function $\nu: \mathbb R \to \mathbb R_+$, two subsets $X,
X' \subset Y$ have \textit{$\nu$-bounded intersection} if the
following inequality holds
$$\diam{N_U (X) \cap N_U (X')} < \nu(U)$$
for any $U \geq 0$.
\end{defn}
\begin{rem}
Typical examples include sufficiently separated quasiconvex subsets
in hyperbolic spaces, and parabolic cosets in relatively hyperbolic
groups(see \cite{DruSapir}).
\end{rem}

A $(\mu, \epsilon)$-contracting system $\mathbb X$ is said to have
\textit{$\nu$-bounded intersection} if any two distinct $X, X' \in
\mathbb X$ do so. A related notion is the following bounded
projection property, which is equivalent to the bounded intersection
property under the contracting assumption.

\begin{defn}[Bounded Projection]
Let $X, X' \subset Y$ be two subsets. Then $X$
has\textit{$B$-bounded projection} to $X'$ for some $B > 0$ if the
following holds
$$\diam{\proj_{X'}(X)} < B$$
\end{defn}

\begin{lem}[Bounded intersection $\Leftrightarrow$ Bounded projection]\label{equivalence}
Assume that $X, X'$ are two $(\mu, \epsilon)$-contracting subsets.
Then $X, X'$ have $\nu$-bounded intersection for some $\nu: \mathbb
R \to \mathbb R_+$ if and only if they have $B$-bounded projection
for some $B > 0$.
\end{lem}
\begin{proof}
$\Rightarrow$: Let $z, w \in \proj_X(X')$ be such that $d(z,w) =
\diam{\proj_X(X')}$. Then there exist $\hat z, \hat w \in X'$ which
project to $z, w$ respectively. Set $B = 2\epsilon_{1, 0} +
\nu(\mu_{1,0})$.

Let's consider $d(z, X)> \mu_{1, 0}$ and $d(w, X)> \mu_{1, 0}$.
Other cases are easier. Let $p$ be a geodesic segment between $z,
w$. Let $\hat u, \hat v$ be the first and last points respectively
on $p$ such that $d(\hat u, X) \le \mu_{1, 0}, d(\hat v, X) \le
\mu_{1, 0}$. Let $u, v$ be a projection point of $\hat u, \hat v$ to
$X$ respectively. Then $d(u, v) \le \nu(\mu_{1,0})$ by the
$\nu$-bounded intersection of $X, X'$. Since $X$ is a $(\mu,
\epsilon)$-contracting subset, we obtain that $d(z, u) \le
\epsilon_{1, 0}, d(w, v) \le \epsilon_{1, 0}$. Hence
$$d(z, w) \le d(z, u) + d(u, v) + d(v, w) \le B.$$

$\Leftarrow$: Given $U >0$, let $z, w \in N_U(X) \cap N_U(X')$. Let
$\hat z, \hat w \in X'$ be such that $d(\hat z, z) \le U, d(\hat w,
z) \le U$. Project $z, w$ to $z', w' \in X$ respectively. Then $d(z,
w) \le d(z', w') + 2U$. We are going to bound the distance $d(z',
w')$.

Observe that the diameter of projection of a geodesic segment of
length $U$ to $X$ is upper bounded by $(2\mu_{1, 0} + \epsilon_{1,
0} + U)$. Hence $\diam{\proj_X([\hat z, z])} \le (2\mu_{1, 0} +
\epsilon_{1, 0} + U), \diam{\proj_X([\hat w, w])} \le (2\mu_{1, 0} +
\epsilon_{1, 0} + U)$. It follows that
$$
\begin{array}{rl}
d(z' ,w') &\le \diam{\proj_X([\hat z, z])} + \diam{\proj_X(X')} +
\diam{\proj_X([\hat w, w])} \\
&\le B +2(2\mu_{1, 0} + \epsilon_{1, 0} + U).
\end{array}
$$

Then $d(z, w) \le B +4\mu_{1, 0} + 2\epsilon_{1, 0} + 2U$. It
suffices to set $\nu(U) = B +4\mu_{1, 0} + 2\epsilon_{1, 0} + 2U$.
\end{proof}

It is proved in \cite{BF2} that a triangle with one side near a
contracting subset is thin. We formulate a similar result here.
\begin{lem}\label{thin}
Let $X$ be a $(\mu, \epsilon)$-contracting set and $x \in X$ and $y
\notin X$ two points. Then there is a constant $\sigma >0$ such that
for any projection point $o$ of $y$ to $X$, we have $d(o, [x, y])
\le \sigma$.
\end{lem}

It is useful to note the following fact with essential same idea as
Lemma \ref{thin}.
\begin{lem}\label{firststep}
Given $\lambda \ge 1, c \ge 0$, let $\gamma = p q$, where $p$ is a
geodesic and $q$ is a $(\lambda,c)$-quasigeodesic in $\mathcal L$.
Assume that $p_-, p_+ \in X \in \mathbb X$ and $q$ has
$\tau$-bounded projection to $X$. Then $\gamma$ is a $(\lambda,
C_{\lambda,c})$-quasigeodesic, where
\begin{equation}\label{C}
C_{\lambda,c} =\lambda(\mu_{\lambda,c}+ \epsilon_{\lambda,c} + \tau)
+ c.
\end{equation}
\end{lem}

\section{Admissible Paths}

This section is is devoted to a notion of an admissible path. The
motivation behind this concept is that we try to give an abstraction
of the concept of a local geodesic in a hyperbolic space, and
establish a generalization for the result in \cite{GH} that a "long"
local geodesic is a global quasi-geodesic. In the absence of
hyperbolicity, our most effort is to deduce such an analogous result
by employing only contracting property in a geodesic metric space.

The terminology of "admissible path" is inspired by the one of
Bestvina-Fujiwara in \cite{BF2}, but possesses quite a different
sense. Roughly speaking, an admissible path can be thought as a
concatenation of quasigeodesics which travels alternatively near
contracting subsets and leave them in an orthogonal way. We
formulate it as follows.
\subsection{Admissible paths}
Recall that $\mathcal L$ is a preferred collection of quasigeodesics
in $Y$ such that $\mathcal L$ contains all geodesics. In what
follows, let $\mathbb X$ be a $(\mu, \epsilon)$-contracting system
in $Y$ with respect to $\mathcal L$. Then each $X \in \mathbb X$ is
$\sigma$-quasiconvex, where $\sigma$ is given by Lemma
\ref{quasiconvexity}.

\begin{defn}[Admissible Path]\label{AdmDef}
Fix a function $\nu: \mathbb R \to \mathbb R_+$ and $\lambda\geq 1,
c \geq 0, \tau>0$. Given $D \ge 0$, a \textit{$(D, \lambda, c, \nu,
\tau)$-admissible path} $\gamma$ is a concatenation of  $(\lambda,
c)$-quasigeodesics in $Y$, associated with a collection $\mathbb Z$
of contracting subsets in $\mathbb X$ satisfying the following.

\begin{enumerate}
\item
Exactly one, say $p_i$, of any two consecutive quasigeodesics in
$\gamma$ has two endpoints in $X_i \in \mathbb Z$,

\item
Each $p_i$ has length bigger than  $\lambda D + c$, except that
$p_i$ is the first or last quasigeodesic in $\gamma$,

\item
For each $X_i$,  a quasigeodesic in $\gamma$ with only one endpoint
in $X_i$ lies in $\mathcal L$ and has $\tau$-bounded projection to
$X_i$, and

\item
Either any two $X_i, X_{i+1}$(if defined) have $\nu$-bounded
intersection, or the quasigeodesic, say $q_{i+1}$, between them has
length bigger than $\lambda D + c$.
\end{enumerate}
\end{defn}
\begin{rem}
If $\mathbb X$ has $\nu$-bounded intersection, then the condition
(4) is always fulfilled. By definition, it is not required that $X_i
\ne X_j$ for $i \ne j$. This often facilitates the verification of a
path being admissible.
\end{rem}

For definiteness in the sequel, we usually write $\gamma = p_0 q_1
p_1 \ldots q_n p_n$ and assume that $p_i$ has endpoints in a
contracting subset $X_i \in \mathbb Z$ and the following conditions
hold.

\begin{enumerate}
\item
$\ell(p_{i}) > \lambda D + c$ for $0 < i < n$.
\item
$q_i \in \mathcal L$ has $\tau$-bounded projection to both $X_{i-1}$
and $X_i$ for $1\leq i \leq n$.
\item
For $1\leq i \leq n$, either $\len(q_i)
> \lambda D + c$, or $X_{i-1}$ and $X_i$ have $\nu$-bounded
intersection.
\end{enumerate}

\begin{defn}[Fellow Traveller]\label{Fellow}
Assume that $\gamma = p_0 q_1 p_1 ... q_n p_n$ is an admissible
path, where each $p_i$ has two endpoints in $X_i \in \mathbb X$. Let
$\alpha$ be a path such that $\alpha_- = \gamma_-,
\alpha_+=\gamma_+$.

Given $R >0$, the path $\alpha$ is a \textit{$R$-fellow traveller}
for $\gamma$ if there exists a sequence of successive points $z_i,
w_i$($0 \le i \le n$) on $\alpha$ such that $d(z_i, w_i) \ge 1$ and
$d(z_i, (p_{i})_-) < R, \;d(w_i, (p_{i})_+) < R.$
\end{defn}

\subsection{Quasi-geodesicity of long admissible paths}
The aim of this subsection is to show that  a $(D, \lambda, c, \nu,
\tau)$-admissible path is a quasigeodesic for any sufficiently large
$D$.

\begin{prop}\label{admissible}
Fix a function $\nu: \mathbb R \to \mathbb R_+$ and $\lambda \geq 1,
c \geq 0, \tau>0$. There are constants $D=D(\lambda, c, \nu, \tau)
>0, R = R(\lambda, c, \nu, \tau)>0$ such that the following
statement holds.

Let $\gamma$ be a $(D_0, \lambda, c, \nu, \tau)$-admissible path for
$D_0 > D$. Then any geodesic $\alpha$ between $\gamma_-$ and
$\gamma_+$ is a $R$-fellow traveller for $\gamma$.
\end{prop}

The main corollary is that a "long" admissible path is a
quasigeodesic.

\begin{cor}\label{maincor}
There are constants $D=D(\lambda, c, \nu, \tau)
>0, \Lambda = \Lambda(\lambda, c, \nu, \tau)>1$ such that given $D_0
>D$ any $(D_0, \lambda, c,\nu, \tau)$-admissible path is a $(\Lambda,
0)$-quasigeodesic.
\end{cor}
\begin{proof}
Let $D , R$ be given by Proposition \ref{admissible}. It suffices to
set $\Lambda = \lambda(6R + 1) + 3c$ to complete the proof.
\end{proof}

The reminder of this subsection is devoted to the proof of
Proposition \ref{admissible}.

Let $\gamma = p_0 q_1 p_1 \ldots q_n p_n$ be a $(D_0, \lambda,
c)$-admissible path for $D_0 > D$, where $p_i, q_i$ are $(\lambda,
c)$-quasigeodesics. For definiteness, assume that each $p_i$ has two
endpoints in $X_i \in \mathbb Z$. Moreover, assume that each $p_i$
is a geodesic, as the general case follows as a direct consequence.
The proof is achieved by the induction on the number of contracting
subsets $\mathbb Z=\{X_i: 0\le i\le n \}$ for $\gamma$.

We define, a priori, the candidate constants which are in fact
calculated in the course of proof:
$$
R = R(\lambda, c) =  \max \{(\ref{R1}), \;(\ref{R3})),
$$
and
$$
D = D(\lambda,c) =  \max\{(\ref{D1}),\; (\ref{D2})\;, (\ref{D4}),
\;(\ref{D5})\}
$$

We start with the lemma below, saying that two consecutive
quasigeodesics of $\gamma$ around $X_k$ have bounded projection to
$X_k$. For simplicity we denote $\proj_k (q)=\proj_{X_k} (q)$.

\begin{lem}\label{neartarget}
Let $X_k$$(0 \le k \le n)$ be a contracting subset for $\gamma$.
Then we have the following
$$\forall k >0: \diam {\proj_k(p_{k-1}  q_{k})} < B_{\lambda,c}, $$
and
$$\forall k < n: \diam {\proj_k(q_{k+1}  p_{k+1})} < B_{\lambda,c},$$
where
$$
B_{\lambda,c} = 2\epsilon_{1,0} + 2\mu_{1,0} + \nu(\mu_{1,0} +
\sigma_0) + \tau.
$$
\end{lem}

\begin{proof}
We only prove the inequality for the case $p_{k-1}q_{k}$. The other
case is similar. We claim the following inequality
\begin{equation}\label{Q1}
\diam {p_{k-1} \cap N_{\mu(1,0)}(X_k)} < \nu(\mu_{1,0} + \sigma_0) ,
\end{equation}
from which the conclusion follows. In fact, assuming the inequality
(\ref{Q1}) is true. Let $z$(resp. $w$) be the first(resp. last)
point of $p_{k-1}$ such that $z, w \in N_{\mu(1,0)}(X_k).$ Then we
have
$$
\begin{array}{rl}
\diam {\proj_k(p_{k-1}  q_k)} &<\diam {\proj_k([(p_{k-1})_-,
z]_{p_{k-1}})} + \diam {\proj_k([z,
w]_{p_{k-1}})} \\
& + \diam {\proj_k([w, (p_{k-1})_+]_{p_{k-1}})} + \diam
{\proj_k(q_k)}\\
&<2\epsilon_{1,0} + (2\mu_{1,0} + \nu(\mu_{1,0} + \sigma_0)) + \tau
\le B_{\lambda,c}.
\end{array}
$$

In order to prove (\ref{Q1}), we examine the following two cases by
the definition of an admissible path.

\textit{Case 1:} $\len (q_k) > \lambda D + c$. We will show that
$p_{k-1} \cap N_{\mu(1,0)}(X_k) = \emptyset$ and hence (\ref{Q1})
holds trivially. Suppose not. Let $w$ be the last point on $p_{k-1}$
such that $d(w, X_k) \le \mu_{1,0}$. Project $w$ to a point $w' \in
X_k$. Then $d(w, w')< \mu_{1,0}$. Using projection, we obtain
$$\begin{array}{rl}
d(w, (q_k)_+) &< d(w, w') + d(w', (q_k)_+)\\
&< \mu_{1,0} + \diam{\proj_k [w, p_{k-1}]_{p_{k-1}}} + \diam{\proj_k (q_k)}\\
&< \mu_{1,0} + \epsilon_{1,0} + \tau.
\end{array}$$ Since $p_{k-1}
q_k$ is a $(\lambda, C_{\lambda,c})$-quasigeodesic by Lemma
\ref{firststep}, we have that
$$C_{\lambda,c} + \lambda d(w, (q_k)_+)> \len([w, (p_{k-1})_+]_{p_{k-1}}) + \len(q_k).$$
As it is assumed that
\begin{equation}\label{D1}
D> \mu_{1,0} + \epsilon_{1,0} + \tau+ C_{\lambda,c},
\end{equation}
this gives a contradiction with $\len(q_k) > \lambda D + c$.

\textit{Case 2:} Otherwise, $X_{k-1}, X_k$ have $\nu$-bounded
intersection. Then $p_{k-1}$ lies in $N_{\sigma(0)}(X_{k-1})$. By
the bounded projection property, we have $$\diam {p_{k-1} \cap
N_{\mu(1,0)}(X_k)} < \diam {N_{\sigma(0)} (X_{k-1}) \cap
N_{\mu(1,0)}(X_k)} < \nu(\mu_{1,0} + \sigma_0).$$ This establishes
(\ref{Q1}).
\end{proof}

We now prove the base step of induction.
\begin{lem}[Base Step]\label{basestep}
Proposition \ref{admissible} is true for $n=1$ and $n=2$.
\end{lem}
\begin{proof}
We shall prove a slightly stronger result: if $\alpha$ is a geodesic
such that $d(\alpha_-, \gamma_-) \le \mu_{1,0},\; d(\alpha_+,
\gamma_+) \le \mu_{1,0}$, then $\alpha$ is a $R$-fellow traveller
for $\gamma$.

\textbf{The case "$n=1$"}. Assume that $\gamma= q_1 p_1 q_2$, where
the geodesic $p_1$ has two endpoints in $X_1$. Since $d(\alpha_-,
\gamma_-), d(\alpha_+, \gamma_+) < \mu_{1, 0}$, it is easy to verify
the following
$$
\diam{\proj_1([\alpha_-, \gamma_-])} \le \epsilon_{1,0} + 3 \mu_{1,
0}, \; \diam{\proj_1([\alpha_+, \gamma_+])} \le \epsilon_{1,0} + 3
\mu_{1, 0}.
$$

We claim that $N_{\mu(1, 0)}(X_1) \cap \alpha \neq \emptyset$.
Suppose not. We estimate the length of $p_1$ by projection
$$
\begin{array}{rl}
\len (p_1) & \le \diam {\proj_1(q_1)} + \diam {\proj_1(\alpha)}
+\diam {\proj_1(q_2)} + \diam{\proj_1([\alpha_-, \gamma_-])} +
\diam{\proj_1([\alpha_+, \gamma_+])} \\
& \le 2\tau + 3\epsilon_{1, 0} + 6 \mu_{1, 0}.
\end{array}
$$
This gives a contradiction as it is assumed that
\begin{equation}\label{D2}
D > 2\tau + 3\epsilon_{1, 0} + 6 \mu_{1, 0}.
\end{equation}

Let $z$ and $w$ be the first and last points of $\alpha$ such that
$z, w \in N_{\mu(1, 0)}(X_1).$ Project $z, w$ to $z', w' \in X_1$
respectively. Hence we see
$$
\begin{array}{rl}
d((q_1)_+, z) &\le \diam {\proj_1(q_1)} + \diam {\proj_1([\alpha_-,
z]_\alpha)} + \diam{\proj_1([\alpha_-, \gamma_-])} + d(z, z') \\
& \le A_{\lambda, c} + 2\epsilon_{1, 0} + 4\mu_{1, 0} < R -1,
\end{array}
$$ as it is assumed that
\begin{equation}\label{R1}
R > \tau + 2\epsilon_{1, 0} + 4\mu_{1, 0} + 1.
\end{equation}
It is similar to prove that $d((q_2)_-, w) < R -1$. Up to a slight
modification of $z, w$, we see that $\alpha$ is a $R$-fellow
traveller for $\gamma$.

\textbf{The case "$n=2$"}. Assume that $\gamma = q_1 p_1 q_2 p_2
q_3$, where the geodesic $p_1, p_2$ have two endpoints in $X_1, X_2$
respectively. The proof is analogous to the above one. We leave it
to the interested reader.
\end{proof}

\textbf{Inductive Assumption}: Assume that Proposition
\ref{admissible} holds for any $(D_0, \lambda,c,\nu,
\tau)$-admissible path $\gamma$ with $k\le n$ contracting subsets
$X_i \in \mathbb X$. Any geodesic between $\gamma_-, \gamma_+$ is a
$R$-fellow traveller for $\gamma$. Thus, $\gamma$ is a
$(\Lambda,0)$-quasigeodesic, where $\Lambda = \Lambda(\lambda, c,
\nu, \tau)$ is given by Corollary \ref{maincor}.

We consider an admissible path $\gamma = p_0 q_1 p_1 \ldots q_n
p_n$, with $(n+1)$ contracting subsets $\mathbb Z=\{X_i \in \mathbb
X: 0 \le i \le n\}$.

The next lemma, compared with Lemma \ref{neartarget}, shows that the
subpath concatenated by $\{p_i, q_i: i < k-1\}$ of $\gamma$ is far
from $X_k$.
\begin{lem}\label{fartarget}
Let $X_k$ $(0 <k < n)$ be a contracting subset for $\gamma$. Denote
$\beta = [(p_0)_-,(q_{k-1})_+]_\gamma$. Then the following holds
$$\beta \cap N_{R+ \sigma(\mu(1,0))}(X_k) = \emptyset.$$
\end{lem}
\begin{proof}
Suppose, to the contrary, that $\beta \cap N_{R+
\sigma(\mu(1,0))}(X_k) \ne \emptyset$. Let $z$ be the last point on
$\beta$ such that $d(z, X_k) \leq R + \sigma_{\mu(1,0)}$, and
project $z$ to $w$ on $X_k$.

Observe that $\hat \beta = \beta  p_{k-1} q_k [(p_k)_-, w]$ is a
$(D_0, \lambda,c)$-admissible path with at most $n$ contracting
subsets, as $k < n$. By Inductive Assumption, $\hat \beta$ is a
$(\Lambda, 0)$-quasigeodesic. It follows that
$$\len([z, w]_{\hat \beta}) < \Lambda(R + \sigma_{\mu(1,0)}).$$ This
gives a contradiction with $\len (p_{k-1})
> D$, as it is assumed that
\begin{equation}\label{D4}
D > \Lambda(R  + \sigma_{\mu(1,0)}).
\end{equation}
Therefore,  $\beta$ has at least a distance $R + \sigma_{\mu(1,0)}$
to $X_k$.
\end{proof}

\begin{figure}[htb]
\centering \scalebox{0.7}{
\includegraphics{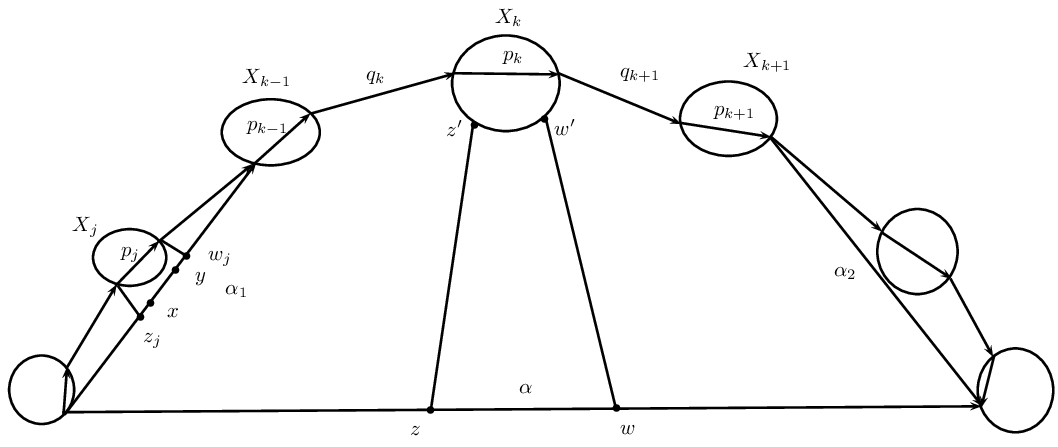} 
} \caption{Proof of Proposition \ref{admissible}} \label{fig:fig1}
\end{figure}

A key step in proving Proposition \ref{admissible} is the following.
\begin{lem}\label{Rnear}
Let $X_k$ be a contracting subset for $\gamma$, where $0 < k < n$.
Assume $\alpha$ is a geodesic such that $d(\gamma_-, \alpha_-) \le
\mu_{1, 0},\; d(\gamma_+, \alpha_+) \le \mu_{1, 0}$. Then  there
exist points $z, w \in \alpha \cap N_{\mu(1, 0)}(X_k)$ such that
$d(z, (p_k)_-) < R-1, \; d(w, (p_k)_+) < R-1$.
\end{lem}

\begin{proof}
Consider $\alpha_1 = [\gamma_-, (p_{k-1})_-]$, $\alpha_2 =
[(p_{k+1})_+, \gamma_+]$, $\beta_1 = [\gamma_-, (p_{k-1})_-]_\gamma$
and $\beta_2 = [(p_{k+1})_+, \gamma_+]_\gamma$. Note that $\alpha_1,
\alpha_2, \beta_1, \beta_2$ may be trivial.

We can apply Induction Assumption to $\beta_1$, as there are at most
$n$ contracting sets associated to $\beta_1$. Thus, $\alpha_1$ is a
$R$-fellow traveller for $\beta_1$. Let $z_j, w_j \in \alpha_1$,
$0\le j <k-1$ be the points given by Definition \ref{Fellow}. Then
$d(z_j, w_j) \ge 1$ and $d((p_j)_-, z_j) < R$, $d((p_j)_+, w_j) <
R$.

An additional difficulty arises here, as we want to get a uniform
bound in (\ref{L}), (\ref{R}) below. But using projections of
quasigeodesics $\beta_i$ to $X_k$ are not sufficient, because
firstly $\beta_i$ may not be in $\mathcal L$, and even this is true,
the projection will force us to increase $D, R$ in each step of
induction. So it is tempting to project the geodesics $\alpha_i$ to
$X_k$. For this purpose, we have the following Claim.

Let $x, y$ be the first and last points on $\alpha_1$ respectively
such that $x, y \in N_{\mu(1,0)}(X_k)$. See Figure \ref{fig:fig1}.

\begin{claim}
The segment $[x, y]_{\alpha_1}$ contains no points from $\{z_j, w_j:
0\le j < k-1\}$. Moreover, any geodesic segment  $[(p_j)_-, z_j]$
and $[(p_j)_+, w_j]$ have at least a distance $\mu_{1,0}$ to $X_k$.
\end{claim}
\begin{proof}[Proof of Claim]
Suppose, to the contrary, that there is a point, say $z_j$, from
$\{z_j, w_j: 0\le < j < k-1\}$ such that $z_j \in [x,
y]_{\alpha_1}$. Note that any point on $[x, y]_{\alpha_1}$ has at
most a distance $\sigma_{\mu(1,0)}$ to $X_k$. As $d(z_j, (p_j)_-) <
R$, we have $d((p_j)_-, X_k) < R + \sigma_{\mu(1,0)}$. This is a
contradiction, as by Lemma \ref{fartarget}, $\beta_1$ has at least a
distance $(R + \sigma_{\mu(1,0)})$ to $X_k$.

By the same argument, one can see that any geodesic segment
$[(p_j)_-, z_j]$ and $[(p_j)_+, w_j]$ have at least a distance
$\mu_{1,0}$ to $X_k$.
\end{proof}
By the Claim above, we assume that $x, y \in [z_j, w_j]_\alpha$ for
some $z_j, w_j \in \{z_j, w_j: 1 < j < k-1\}$. (The case that $x, y
\in [w_j, z_{j+1}]_\alpha$ is similar).

Using projection, we see that $\alpha \cap N_{\mu(1,0)} (X_k) \neq
\emptyset$. Suppose not. Projecting the right-hand segment of
$\gamma$ to $X_k$, we get
\begin{equation}\label{L}
\begin{array}{rl}
\diam{\proj_k([\gamma_-, (p_k)_-]_\gamma)} & <  \diam{\proj_k([\alpha_-, z_j]_{\alpha_1})} + \diam{\proj_k ([(p_j)_-, z_j])} \\
& + \diam{\proj_k (p_j)} + \diam{\proj_k [(p_j)_+, w_j]}  \\
& + \diam{\proj_k ([w_j, (\alpha_1)_+]_{\alpha_1})} + \diam {\proj_k (p_{k-1} q_k)}  \\
&< 5\epsilon_{1,0} + B_{\lambda,c}.
\end{array}
\end{equation}
The same estimation applies to the left-hand segment,
\begin{equation}\label{R}
\diam{\proj([(p_k)_+, \gamma_+]_\gamma)} < 5\epsilon_{1,0} +
B_{\lambda,c}.
\end{equation}
Since $[\alpha_-, \gamma_-]$ and $[\alpha_+, \gamma_+]$ are of
length at most $\mu_{1, 0}$, we have
\begin{equation}\label{M}
\diam{\proj_k([\alpha_-, \gamma_-])} \le \epsilon + 3 \mu_{1, 0}, \;
\diam{\proj_k([\alpha_+, \gamma_+])} \le \epsilon + 3 \mu_{1, 0}.
\end{equation}
By (\ref{L}), (\ref{M}) and (\ref{R}), the length of $p_k$ is
estimated as follows
$$\begin{array}{rl} \len(p_k) &< \diam{\proj_k([\gamma_-,
(p_k)_-]_\gamma)} + \diam{\proj_k([(p_k)_+, \gamma_+]_\gamma)} +
\diam{\proj_k(\alpha)} \\
& + \diam{\proj_k([\alpha_-, \gamma_-])} + \diam{\proj_k([\alpha_+, \gamma_+])} \\
& < 13\epsilon_{1,0} + 6\mu_{1, 0} + 2B_{\lambda,c}.
\end{array}$$
This gives a contradiction with $\len(p_k) > D$, as it is assumed
that
\begin{equation}\label{D5}
D > 13\epsilon_{1,0} + 6\mu_{1, 0} + 2B_{\lambda,c}.
\end{equation}

Hence, $\alpha \cap N_{\mu(1,0)} (X_k) \neq \emptyset$. Let $z$ be
the first point of $\alpha$ such that $d(z, X_k) \leq \mu_{1,0}$,and
$z'$ a projection point of $z$ to $X_k$. Thus
$$
\begin{array}{rl}
d(z, (p_k)_-) &< d(z,z') + d(z', (p_k)_-) \\
&< \mu_{1,0} + 5\epsilon_{1,0} + B_{\lambda,c} < R -1,
\end{array}$$
as it is assumed that
\begin{equation}\label{R3}
R > \mu_{1,0} + 5\epsilon_{1,0} + B_{\lambda,c} +1.
\end{equation}

Let $w$ be the last point of $\alpha$ such that $d(z, X_k) \leq
\mu_{1,0}$. Arguing with the same line for the point $z$, we see
that $d(w, (p_k)_+) < R -1$. This completes the proof.
\end{proof}

We are now in a position to finish the proof of Proposition
\ref{admissible}, with repeated applications of Lemma \ref{Rnear}.

\begin{proof}[Proof of Proposition \ref{admissible}]
Recall that $\gamma$ is a $(D, \lambda, c, \nu, \tau)$-admissible
path with at most $(n+1)$ contracting subsets. Consider a geodesic
$\alpha$ with $\alpha_- = \gamma_-$ and $\alpha_+ = \gamma_+$. By
Lemma \ref{basestep}, we can assume that $n \ge 2$.

Consider a contracting subset $X_k$ for $\gamma$, where $0 < k < n$.
By Lemma \ref{Rnear}, there exist $z_k, w_k \in \alpha \cap
N_{\mu(1, 0)}$ such that $d(z_k, (p_k)_-) < R-1,\; d(w_k, (p_k)_+) <
R-1$.

Let $w_k'$ be a projection point of $w_k$ to $X_k$. We consider $j =
k + 1$. Denote $\gamma' = [w_k', (p_k)_+][(p_k)_+, \gamma_+]_\gamma$
and $\alpha' = [w_k, \alpha_+]_\alpha$.

Observe that $\gamma'$ is a $(D, \lambda, c, \nu, \tau)$-admissible
path with at most $n$ contracting subsets. Since $d(w_k, w_k') \le
\mu_{1, 0}$, we apply Lemma \ref{Rnear} to $\gamma'$ and $\alpha'$.
Then there exist points $z_j, w_j \in \alpha' \cap
N_{\mu(1,0)}(X_j)$ such that $d(z_j, (p_j)_-) <R-1, d(w_j, (p_j)_+)
< R-1$.

Continuously increasing or decreasing $k$, the points $z_j, w_j$ on
$\alpha$ are obtained to satisfy $d(z_j, (p_j)_-) <R-1,\; d(w_j,
(p_j)_+) < R-1$ for all $0 \le j \le n$.

Hence, up to a slight modification of $z_j, w_j$, one sees that
$$d(z_j, w_j)\ge 1,\; d(z_j, (p_j)_-) <R, \; d(w_j, (p_j)_+) < R.$$
So, $\alpha$ is a $R$-fellow traveller, concluding the proof.
\end{proof}

The following Corollary will be useful in Section 4.
\begin{cor}\label{bddproj}
Let $X_k$ be a contracting subset for $\gamma$, where $0 \le k \le
n$. Denote $\beta =[\gamma_-, (p_k)_-]_\gamma$. Then
$\diam{\proj_k(\beta)} < 5\epsilon_{1,0} + B_{\lambda,c}$.
\end{cor}
\begin{proof}
In order to use Induction Assumption, we limit $0< k < n$ in the
proof of Lemma \ref{Rnear}.  By Proposition \ref{admissible}, it is
clear that the same argument works for every $0\le k\le n$. Thus the
conclusion follows.
\end{proof}

\section{Groups with a contracting subgroup}

Assume that $G$ acts properly on $(Y, d)$.  A subgroup $H$ in $G$ is
called \textit{contracting} if for some (hence any) $o \in Y$, the
subset $Ho$ is contracting in $Y$. $H$ is called \textit{strongly
contracting} if for some (hence any) $o \in Y$, the collection
$\{gHo: g\in G\}$ is a contracting system with bounded intersection
in $Y$.

An element $h \in G$ is called (resp, \textit{strongly})
\textit{contracting} if the subgroup $\langle h \rangle$ is (resp.
strongly) contracting.

We describe a typical situation which gives arise to a contracting
system with bounded intersection (and projection) property. The
proof is straightforward and left to the reader.
\begin{lem}\label{contractingsystem}
Suppose that a group $G$ acts properly on a geodesic metric space
$(Y, d)$ with a $G$-invariant contracting system $\mathbb X$. Assume
that $\sharp (\mathbb X/G) < \infty$ and for each $X \in \mathbb X$
the stabilizer $G_X$ of $X$ in $G$ acts cocompactly on $X$. If the
following holds
\begin{equation}\label{bddint}
\diam{N_r(X) \cap N_r(X')} < \infty
\end{equation}
for any $r>0$ and all $X \neq X' \in \mathbb X$, then $\mathbb X$
has bounded projection property.
\end{lem}

A contracting subgroup always produces a strongly contracting
subgroup. Let $H$ be contracting. Consider a group $E(H)$ defined as
follows:
$$E(H)=\{g: \exists r >0, \diam{N_r(gHo) \cap N_r(Ho)} = \infty \}.$$
We note the following facts.
\begin{lem}\label{maxelem}
$[E(H): H]< \infty$. In particular, $E(H)$ is strongly contracting.
\end{lem}
\begin{proof}
In definition of $E(H)$, it is easy to see that $r>0$ can be made
uniform for all $g$ by contracting property. Since $G$ acts properly
on $Y$, we then deduce that $H$ is of finite index in $E(H)$. By
Lemma \ref{contractingsystem}, we see that  $E(H)$ is contracting.
\end{proof}

\begin{lem}\label{samecoset}
If $gE(H)o=g'E(H)o$ for some $g, g' \in G$, then $gE(H)=g'E(H)$.
\end{lem}
\begin{proof}
By Lemma \ref{maxelem}, there exists a constant $D>$ such that
$$gE(H)o\subset N_D(gHo),\; g'E(H)o\subset N_D(g'Ho).$$ This proves
that $g^{-1}g' \in E(H)$.
\end{proof}

Before going on, we discuss a notion of free products of sets
defined in \cite{DPPS}. For a subset $A$ and an element $h$ in $G$,
the \textit{free product} $\mathcal W(A, h)$ of sets $A$ and $h$ is
the set of all words $W$ in the alphabet $A \cup \{h\}$ such that
exactly one of two consecutive letters in $W$ is either $h$ or in
the set $A$. Let $$\kappa: \mathcal W(A, h) \to G$$ be the
evaluation map with its image in $G$ denoted by $A \star h$.

Take a word $W \in\mathcal W(A, h)$ to which we associate a normal
path as follows. Denote by $W_i$ the initial subword of $W$ of
length $i$ for $1 \le i\le \len(W)$. For a basepoint $o\in Y$, set
$o(W):=\{o_i: o_i=\kappa(W_i)o, 1 \le i\le \len(W), o_0=o\}$ in $Y$.
The \textit{normal path} of $W$ is thus the concatenation of
geodesics connecting $o_i, o_{i+1}$ for $0\le i < \len(W)$.

A contracting element by definition gives a contracting subgroup,
however, in which elements might not be contracting. For example, a
maximal parabolic subgroup in a relatively hyperbolic group usually
contains non-contracting elements. The following technical lemma
gives a way to construct contracting elements in a free product of a
given contracting subgroup and some element.

\begin{lem}\label{manycontr}
Suppose $G$ acts properly on $(Y, d)$ with a contracting subgroup
$H$. For any $k \in G\setminus E(H)$ and $o\in Y$, there exists
$D=D(k, o)>0$ with the following property. Choose a finite set $F
\subset E(H)$ such that $d(o, fo)>D$ for all $f\in F$. Then for any
$W\in \mathcal W(F, k)$, the set $o(W)$ is contracting with
contracting constant depending on $F$.
\end{lem}

\begin{proof}
Since $E(H)$ is strongly contracting, $\mathbb X=\{gE(H)o: g\in G\}$
is $(\mu, \epsilon)$-contracting with $\nu$-bounded intersection for
some fixed $\mu, \epsilon, \nu$. Let $D=\min\{d(o, fo: f\in F)\}$.

Assume that $W=h_0 k h_1 k \cdots h_n k\in \mathcal W(F, k)$ for
$n>0, h_i\in F$. Consider the normal path $\gamma=p_0q_1p_1\ldots
p_{n-1} q_n$, where $p_i$ are geodesics corresponding to $h_i$ and
$q_i$ geodesics corresponding to $k$. Each $p_i$ has two endpoints
in some $X_i \in \mathbb X$. By a translation of $q_i, q_{i+1}, X_i$
to a standard position, we see that there exists an uniform number
$\tau>0$ such that $q_i$ and $q_{i+1}$ have $\tau$-bounded
projection to $X_i$. Hence, $\gamma$ is a $(D,1,0,\nu, \tau)$
admissible path. Let $D_0=D(1, 0,\nu,\tau), \Lambda=\Lambda(1,
0,\nu,\tau)$ given by Proposition \ref{admissible}. By choosing
$D>D_0$, we obtain that $\gamma$ is a $(\Lambda, 0)$-quasigeodesic.

Denote $A=o(W)$. Let $\alpha$ be a geodesic such that $\alpha \cap
N_{\epsilon(1,0)}(A) =\emptyset$. Let $z, w \in A$ be projection
points of $\alpha_-, \alpha_+$ to $A$. Without loss of generality,
assume that $d(z,w)=\diam{\proj_A(\alpha)}$. Construct a normal path
$\gamma$ between $z$ and $w$.

To prove the contracting property of $A$, we shall show that
$\gamma$ contains at most two geodesics from $\{p_i\}$. Clearly,
this implies that $A$ is $(\epsilon_{1,0}, D_1)$-contracting, where
$D_1=3d(o, ko)+2\max\{d(o, fo): f\in F\}$.

Suppose to the contrary that $\gamma$ contains three distinct $p_i,
p_j, p_k$ in a natural order. Thus $[z,\alpha_-]\cap
N_{\epsilon(1,0)}(X_j)=\emptyset$. Indeed, if not, let $x\in [z,
\alpha_-]$ be the first point such that $x\in
N_{\epsilon(1,0)}(X_j)$, and $x'\in X_j$ a projection point of $x$
to $X_j$. By Corollary \ref{bddproj}, we use projection to see that
\begin{equation}\label{P1}
d((p_j)_-, x) <\diam{\proj_{X_j}([z,
(p_j)_-]_\gamma)}+\diam{\proj_{X_j}([z, x])}+d(x,x')<6\epsilon_{1,0}
+\mu_{1,0}+ B_{1,0}.
\end{equation}

Note that $\gamma$ is a $(\Lambda, 0)$-quasigeodesic, and contains
$p_i$. Choose further $D>2\Lambda(6\epsilon_{1,0} +\mu_{1,0}+
B_{1,0})$. It follows that
\begin{equation}\label{P2}
d(z, (p_j)_-)>\Lambda^{-1}d(o, ho)> 2(6\epsilon_{1,0} +\mu_{1,0}+
B_{1,0}).
\end{equation}
Combining (\ref{P1}) and (\ref{P2}), we get the following
$$d(x, z)>d(z, (p_j)_-) - d((p_j)_-, x)>d((p_j)_-, x).$$ As
$z$ is a projection point from $\alpha$ to $A$ and $(p_j)_- \in A$,
this gives a contradiction.

By the same reasoning, we see that $\alpha \cap N_{\epsilon(0,
1)}(X_j)=\emptyset$. However, if we project $[z, \alpha_-], \alpha,
[w, \alpha_+]$ to $X_j$, we have by Corollary \ref{bddproj} that
$$\len(p_j) < 3\mu_{1,0})+\diam{\proj_{X_j}([z, (p_j)_-]_\gamma)} +
\diam{\proj_{X_j}([w, (p_j)_+]_\gamma)} <10\epsilon_{1,0}
+5\mu_{1,0}+ 2B_{1,0}.$$ We would get a contradiction by setting
$D>10\epsilon_{1,0} +5\mu_{1,0}+ 2B_{1,0}$. Therefore, $\gamma$
contains at most two geodesics from $\{p_i\}$. The proof is thus
complete.
\end{proof}

\begin{cor}\label{manycontr}
Suppose $G$ acts properly on $(Y, d)$ with a contracting subgroup
$H$. Consider $k \in G\setminus E(H)$ and $o\in Y$. Then there
exists $D=D(k, o)>0$ such that for any $h\in E(H)$ with $d(o, ho)
>D$ the element $hk$ is contracting in $G$.
\end{cor}

\begin{lem}\label{normal}
Suppose a non-elementary group $G$ acts properly on $(X, d)$ with a
contracting element. Let $\Gamma$ be an infinite normal subgroup.
Then $\Gamma$ contains infinitely many contracting elements.
\end{lem}

\begin{proof}
Let $h$ be a contracting element in $G$. Suppose that $h \notin
\Gamma$.  As $\Gamma$ is infinite and $G$ is non-elementary, there
exists infinitely many $k \in \Gamma\setminus E(h)$. Indeed, if not,
assume that $\Gamma \subset E(h)\cdot F$ for a finite set $F \subset
G$. As $G$ is non-elementary, there exists $g \notin E(h)$. Clearly,
$\Gamma=g\Gamma g^{-1} \subset gE(h)Fg^{-1}$. By Lemma
\ref{maxelem}, $gE(h)o$ and $E(h)o$ have bounded intersection.
However, as $\Gamma$ is infinite, we arrive at a contradiction.
Thus, there exists infinitely many $k \in \Gamma\setminus E(h)$.

Fix some $k \in \Gamma\setminus E(h)$. By Lemma \ref{manycontr},
$h^nkh^{-n}k$ is contracting for all sufficiently large $n$.
\end{proof}

\begin{cor}
Existence of an contracting element is inherited by an infinite
normal subgroup of a non-elementary group.
\end{cor}

\section{A critical gap criterion}
This section shall present a critical gap criterion, due to
Dalbo-Peigne-Picaud-Sambusetti, originally stated for Kleinian
groups.  For our purpose, their criterion is formulated in a general
manner.

\subsection{Poincar\'e series}
Suppose $G$ acts properly on a geodesic metric space $(Y, d)$.  Fix
a basepoint $o\in Y$. For any subset $X \subset G$, we can introduce
the following series called Poincar\'e series,
$$
\PS{X} = \sum\limits_{g \in X} \exp(-s\cdot d(o, go)), s\ge 0.
$$
The \textit{critical exponent} $\delta_X$ of $\PS{X}$ is the limit
superior $$\limsup\limits_{n \to \infty} n^{-1} \ln \sharp (B(o, n)
\cap X),$$ which can be thought of the exponential growth rate of
the orbit $Xo$. In particular, for a finite generating set $S$, this
number $\delta_G$ coincides with the usual exponential rate of $G$
relative to $S$, where $G$ acts on its Cayley graph.

The following observation in \cite{DPPS} enables us to consider
growth rate of a net in $G$ rather than the group itself.

\begin{lem}\label{convradius}
Let $A$ be a subset in $G$ such that $N_R(Ao) = Go$ for $R>0$. Then
$\delta_{G} = \delta_{A}$.
\end{lem}

Let $\Gamma$ be a normal subgroup and $\bar Y=Y/\Gamma$ endowed with
the quotient metric $\bar d$. Define $\bar d(\bar x, \bar y) =
d(\Gamma x, \Gamma y)$ for any $\bar x, \bar y \in \bar Y$.

Consider the quotient $\bar G = G/\Gamma$, and the epimorphism $\pi:
G \to \bar G$. Clearly, $\bar G$ acts properly and isometrically on
$\bar X$ by $g\Gamma \cdot \Gamma o = \Gamma go$. Thus, we can also
consider the Poincar\'e series $\PS{\bar G}$ for the action $\bar G$
on $\bar X$, with critical exponent $\g{\bar G}$.

Recall that $G \act Y$ is \textit{geometric} if $G$ acts properly
and cocompactly on $Y$. The following is also essentially proven in
\cite{DPPS}.
\begin{lem} \label{divquotient}
Suppose $G$ acts geometrically on a geodesic metric space $(Y, d)$.
Then for any quotient $\bar G$, the Poincar\'e series $\PS{\bar G}$
is divergent at $s=\g{\bar G}$.
\end{lem}
\begin{proof}
This is proven exactly by the argument in \cite[Proposition
4.1]{DPPS} for the convex-compact case, where the only condition
used is the cocompact action of $G$ on $Y$.
\end{proof}

\subsection{A critical gap criterion}
In \cite{DPPS}, Dalbo-Peigne-Picaud-Sambusetti introduced the
following critical gap criterion in terms of free product of sets.
\begin{lem}\label{degrowth}
Let $A$ be a subset in $G$ and $h$ be a non-trivial element in $G
\setminus A$. Assume that $\kappa: \mathcal W(A, h) \to G$ is
injective. If the series $\PS{A}$ is divergent at $s=\delta_{A}$,
then $\delta_{A \star h}
> \delta_{A, S}$.
\end{lem}

In the sequel, the subset $A$ shall come from (a subset of) the
image of a section map $\iota: \bar G \to G$ defined as follows. For
each $\bar g \in \bar G$, choose an element $g \in \pi^{-1}(\bar g)
= g\Gamma$ such that $d(go, o) = d(go, \Gamma o)$. Let $\iota(\bar
g) = g$.

Without loss of generality, we can assume that $\iota(\bar g^{-1}) =
\iota(\bar g)^{-1}$. In fact, let $g \in \pi^{-1}(\bar g) = g\Gamma$
such that $d(go, o) = d(go, \Gamma o)$. It suffices to show that
$d(g^{-1}o, o) = d(g^{-1}o, \Gamma o)$. This is obvious.

Denote by $G_\Gamma =\iota(\bar G)$ the symmetric image. As $d(go,
o) = \bar d(\bar g \bar o, \bar o)$, we see the relation $$\PS{\bar
G} = \PS{G_\Gamma}.$$ Hence by Lemma \ref{divquotient}, the strategy
in proving Theorem \ref{tight} is to construct inside $G$ a free
product of $G_\Gamma$ with an element $h$. The element $h$ will be
chosen to be a contracting element in $\Gamma$.

\section{Proof of Theorem \ref{tight}}
Suppose $G$ acts geometrically on a geodesic metric space $(Y, d)$.
Let $\Gamma$ be a normal subgroup in $G$. Choose a symmetric set
$G_\Gamma \le G$ consisting of minimal representatives in $\bar G =
G /\Gamma$ as in Section 5.

Let $h$ be a contracting element in $\Gamma$. (Such $h$ always
exists by Lemma \ref{normal} when $\Gamma$ is infinite.) By Lemma
\ref{maxelem}, there exists a subgroup $E(h):=E(\langle h\rangle)$
in $G$ which contains $\langle h\rangle$ as a finite index subgroup.

Choose a basepoint $o$ in $Y$ and consider $C(h)=E(h)o$. By
definition of a contracting element,  there are $\mu, \epsilon
> 0$ such that $C(h)$ is a contracting set. Moreover, there is some
function $\nu: \mathbb R \to \mathbb R$ such that $\mathbb X = \{g
C(h): g \in G \}$ is a $(\mu, \epsilon)$-contracting system with
$\nu$-bounded intersection. By Lemma \ref{equivalence}, $\mathbb X$
has $B$-bounded projection for some $B>0$.

\subsection{Constructing a free product of subsets in $G$}
Given $A \subset G_\Gamma$ and $n > 0$, consider a free product
$\mathcal W(A, h^n)$ of sets $A$ and $h^n$.

Let $W = a_1 h^n a_2 h^n \ldots a_k h^n \in \mathcal W(A, h^n)$ for
$a_i \in A$. Fix a basepoint $o\in Y$ and consider its normal path
$\gamma = p_1 q_1 \ldots p_k q_k$, where $p_i$ are geodesics
corresponding to $a_i$, and $q_i$ geodesics corresponding to $h^n$.

We shall prove that every normal path is, in fact, an admissible
path. First of all, we need the following observation.

\begin{lem}\label{orth}
There exists a constant $\tau = \tau(h) > 0$ such that the following
holds:

Given $g \in G_\Gamma$, any geodesic between $o$ and $go$ has
bounded projection $\tau$ to the set $C(h)$.
\end{lem}

\begin{proof}
We project $go$ to a point  $z \in C(h)$. By Lemma \ref{thin}, there
is a constant $\sigma_1 > 0$ such that $d(z, [o, go]) < \sigma_1$.
Choose $w \in [o, go]$ such that $d(z, w) < \sigma_1$.

Since $\langle h \rangle$ is of finite index in $E(h)$, there exists
a constant $\sigma_2>0$ such that $C(h) \subset N_{\sigma_2}(\langle
h \rangle o)$. Thus, there exists $h^i \in \langle h\rangle$ such
that $d(h^io, z) \le \sigma_2$.

Observe that $d(w, o) \le d(w, z)+\sigma_2$. Indeed, suppose to the
contrary that $d(w, o) > d(w, z)+\sigma_2$. Then $d(go, o) = d(go,
w) + d(w, o) > d(go, z)+\sigma_2>d(go, h^io).$ This gives a
contradiction, as $d(o, go) = d(\Gamma o, go)$.

Hence $$d(o, z) \le d(o, w) + d(w,z) \le 2\sigma_1+\sigma_2.$$
Setting $\tau=4\sigma_1+2\sigma_2$ finishes the proof.
\end{proof}

We are now able to prove the following.

\begin{lem}\label{normalpath}
Let $W = a_1 h^n a_2 h^n \ldots a_k h^n$ be a word in $\mathcal W(A,
h^n)$, and $\gamma = p_1 q_1 \ldots p_k q_k$ the normal path of $W$
in $Y$. Then $\gamma$ is a $(D, 1, 0, \nu, \tau)$-admissible path,
where $D=d(o, h^no)$.
\end{lem}
\begin{proof}
Observe that each geodesic $q_i$ has both endpoints in a contracting
subset $$X_i = a_1 h^n a_2 h^n \ldots a_i \cdot C(h) \in \mathbb
X.$$ Clearly, as $\mathbb X$ has bounded projection, it suffices to
prove Condition (3) in Definition \ref{AdmDef}. Namely, we shall
show that each $p_i$($1 < i < k$) has $\tau$-bounded projection to
$X_i$, where $\tau$ is given by Lemma \ref{orth}.  The case that
$p_i$ is $\tau$-orthogonal to $X_{i-1}$(if defined) is similar.

Translate $X_i$ to $C(h)$ by $a_i^{-1} h^{-n} a_{i-1}^{-1} h^{-n}
\ldots a_1^{-1}$. Correspondingly, $p_i$ is translated to a geodesic
between $(a_i)^{-1}o$ and $o$. As $G_\Gamma$ is symmetric, we have
$(a_i)^{-1} \in G_\Gamma$. By Lemma \ref{orth}, any geodesic between
$(a_i)^{-1}o$ and $o$ has a $\tau$-bounded projection to $C(h)$.
Thus, $p_i$ has a $\tau$-bounded projection to $X_i$.
\end{proof}

In order to use Lemma \ref{degrowth}, we need show that $\kappa:
\mathcal W(A, h^n) \to G$ is injective for sufficiently large $n$.

\begin{lem} \label{injective}
There exists $N = N(h) >0$ with the following property.

Let $W = a_1 h^n a_2 \ldots a_k h^n, W'=a_1'h^n a_2' \ldots a_l' h^n
$ be two words in $\mathcal W(A, h^n)$ for $n > N$ such that
$\kappa(W) = \kappa(W')$. Then there is a constant $L = L(h, n) > 0$
we have $d(a_1 \Gamma o, a_1'\Gamma o) < L$.
\end{lem}

\begin{rem}
By Lemma \ref{normalpath} the normal path is admissible and thus is
a quasigeodesic for large $n$. It would be tempting to conclude the
injectivity of $\kappa$ immediately by a general principle to embed
a Schottky group in $G$. However, the non-triviality of this lemma
lies in the fact that $G_\Gamma$ is not a subgroup in $G$. Hence the
cancellation in $W^{-1}W'$ would not necessarily produce a word
again in $\mathcal W(A, h^n)$.
\end{rem}

\begin{proof}
Let $\gamma = p_1 q_1 \ldots p_k q_k, \gamma' = p_1' q_1' \ldots
p_l' q_l'$ be normal paths for $W, W'$ respectively. Possibly,
$k\neq l$. Fix a geodesic $\alpha$ between $o$ and $go$, where
$g=\kappa(W)$.

By Proposition \ref{admissible} and Lemma \ref{normalpath}, there
exist $R = R(h), N=N(h)>0$ such that $\alpha$ is a $R$-fellow
traveller for both $\gamma$ and $\gamma'$. Also take $N$ to satisfy
the following inequality
\begin{equation}\label{N}
d(o, h^no) > 2R + 2M + B + \tau + 10\mu + 3\epsilon,
\end{equation}
for any $n > N$.

Since $h$ is contracting, there is a constant $M > 0$ by Lemma
\ref{maxelem} such that $C(h) \subset N_M(\Gamma o)$.

Consider contracting sets $X = a_1 C(h)$ and $X' =  a_1' C(h)$. We
first consider the case that $X = X'$.  It follows that $d(a_1
\Gamma o, a_1'\Gamma o) < 2M$. Setting $L = 2M$ would finish the
proof. Hence, we assume that $X \ne X'$ in the remainder of proof.

Let $z, w$ be the first and last points on $\gamma$ respectively
such that $z, w \in N_\mu(X) \cap \alpha.$ Since $\alpha$ is a
$R$-fellow traveller for $\gamma$, we have that $d(z, (q_1)_-), d(w,
(q_1)_+) \le R$.  Similarly, choose the first and last points $u, v$
respectively on $\gamma$ such that $u, v \in N_\mu(X') \cap \alpha$
and $d(u, (q_1')_-), d(v, (q_1')_+) \le R$.

We analyze the configuration of pairs $\{z, w\}$ and $\{u, v\}$ as
follows:

\textbf{Case 1).} $[z, w]_\gamma$ and $[u, v]_\gamma$ intersects
nontrivially. For definiteness, assume that $u$ is in $[z,
w]_\gamma$. The other cases are similar.

As $u, w \in N_\mu(X) \cap N_\mu(X')$, we see that $d(u, w) <
\nu(\mu)$ by the $\nu$-bounded intersection of $\mathbb X$. Hence,
$d(a_1 o, a_1' o) < L$, where $L=L(n)$ is defined as follows
$$
L = 2d(h^n o, o) + R + \nu(\mu).$$

\textbf{Case 2).} $[z, w]_\gamma$ and $[u, v]_\gamma$ are disjoint.
Assume that $u, v \in [w, go]_\gamma$. The other case is symmetric.
We shall lead to a contradiction.

Without loss of generality, assume that the geodesic $p_1'$
intersects non-trivially in $N_\mu(X)$. The case that $p_1' \cap
N_\mu(X) = \emptyset$ can be seen as be degenerated in the following
proof. See Figure \ref{fig:fig}.

\begin{figure}[htb] 
\centering \scalebox{0.8}{
\includegraphics{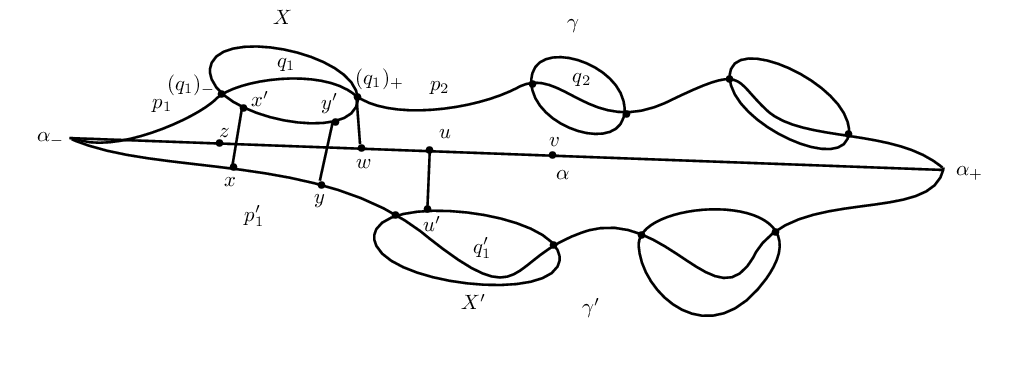} 
} \caption{Proof of Lemma \ref{injective}} \label{fig:fig}
\end{figure}

Let $x, y$ be the first and last points of $p_1'$ such that $x, y
\in N_\mu(X)$. We claim that $d(x, y) \le 2 \mu + 2M$. In fact, as
$a_1' \in G_\Gamma$, we have $d(a_1' o, \Gamma o) =d(a_1'o, o)$.
Then for any $x, y \in [o, a_1'o]$, we have $d(\Gamma x, \Gamma y) =
d(x, y)$. Since $x, y \in N_\mu(X)$ and $X \subset N_M(a_1 \Gamma
o)$, we see that $d(x, a_1 \Gamma o), d(y, a_1 \Gamma o) \le \mu +
M$. This implies that $d(\Gamma x , \Gamma y) \le 2\mu + 2M$. Hence,
$d(x, y) \le 2\mu +2M$.

Project $x, y$ to $x', y' \in X$ respectively. It follows that
\begin{equation}\label{M2}
d(x', y') \le 4\mu + 2M.
\end{equation}

By Lemma \ref{orth} and the fact $[o, x]_{p_1'} \cap N_\mu(X)
=\emptyset$, we obtain the following:
\begin{equation}\label{L2}
d(x', (q_1)_-) \le \diam{\proj_X(p_1)} + \diam{\proj_X([o,
x]_{p_1'})} \le \tau + \epsilon.
\end{equation}

It is easy to see that the projection of a geodesic segment of
length $R$ to $X$ is upper bounded by $R + 2\mu + \epsilon$. Let
$u'$ be a projection point of $u$ to $X'$. Thus,  $\diam{\proj_X([u,
u'])} \le R+2\mu + \epsilon, \;\diam{\proj_X([w, (q_1)_+])}   \le
R+2\mu + \epsilon$. As a consequence, we estimate the following
distance:
\begin{equation}\label{R2}
\begin{array}{rl}
d(y', (q_1)_+) & \le \diam{\proj_X([y, (p_1')_+])} + \diam{\proj_X(X')} + \diam{\proj_X([u, u'])} \\
& + \diam{\proj_X([w, u]_\gamma)}  +
\diam{\proj_X([w, (q_1)_+])} \\
& \le 2\epsilon+B + 2(R + 2\mu + \epsilon) \le 2R + B + 6\mu +
2\epsilon.
\end{array}
\end{equation}

With (\ref{M2}), (\ref{L2}) and (\ref{R2}),  it follows that
$$\len (q_1) < d(x', (q_1)_-) + d(x',
y') + d(y', (q_1)_+) \le 2R + 2M + B + \tau + 10\mu + 3\epsilon.$$
This gives a contradiction with the choice of $N$ in (\ref{N}).
Hence, the case 2) is impossible.
\end{proof}

\subsection{Proof of Theorem \ref{tight}}
Since $\Gamma$ is an infinite normal subgroup, by Lemma \ref{normal}
there are infinitely many contracting elements in $\Gamma$. Fix a
contracting element $h$ in $\Gamma$.

Let $n > N$ be a number and $L=L(h, n)$, where $N=N(h), L=L(h, n)$
are the constants given by Lemma \ref{injective}.

Let $A$ be a subset in $G_\Gamma$ such that for any distinct $a, a'
\in A$,  $d(\Gamma a o, \Gamma a' o)
> L$, and for any $g \in G$, there
exists $a \in A$ such that $d(\Gamma go, \Gamma ao) < L$. By Lemma
\ref{injective}, the evaluation map $\kappa: \mathcal W(A, h^n) \to
G$ is injective.

By construction, $\PS{G_\Gamma} =\PS{\bar G}$. Thus, $\PS{G_\Gamma}$
is divergent at $s = \delta_{\bar G}$ by Lemma \ref{divquotient}. By
Lemma \ref{convradius}, we have $\delta_{A} = \delta_{G_\Gamma} $
and $\PS{A}$ is also divergent at $s=\delta_{A}$.

By Lemma \ref{degrowth}, it follows that $\delta_{G} \ge
\delta_{A\star h^n} > \delta_{A} \ge \delta_{\bar G}$. The proof is
complete.

\section{Groups with non-trivial Floyd boundary}
We recall the construction of Floyd boundary of a finitely generated
group $G$. Let $f: \mathbb N \to \mathbb R_+$ be a \textit{Floyd
function} $f: \mathbb{N} \to \mathbb R_{>0}$ which satisfies the
\textit{summable} condition
$$\sum\limits_{n=1}^{\infty} f(n) < \infty$$
 and the \textit{$\lambda$-delay} property $$\lambda < f(n+1)/f(n) < 1$$ for $\lambda \in ]0, 1[$.

Let $\Gx$ be the Cayley graph of $G$ with respect to a generating
set $S$. Denote by $d_S$ the combinatorial (word) metric on $\Gx $.
The \textit{Floyd length} of each edge $e$ in $\Gx $ is set to be
$f(n)$, where $n = d_S(1, e)$. This naturally defines a length
metric $\rho$ on $\Gx$.  Let $\Gf$ be the Cauchy completion of $G$
with respect to $\rho$. The complement $\pGf$ of $G$ in $\Gf$ is
called \textit{Floyd boundary}. $\pGf$ is \textit{non-trivial} if
$\pGf$ contains at least three points. We refer the reader to
\cite{Floyd}, \cite{Ka} and \cite{GePo4} for more detail.

By construction, the topological type of Floyd boundary depends on
the choice of $f$ and $S$. However, the following lemma roughly says
that the non-triviality of Floyd boundary can be thought of as a
property of the pair $(G, S)$.

\begin{lem}\label{nontri}
Suppose that $G$ has a non-trivial Floyd boundary $\pGf$ for some
finite generating set $S$ and Floyd function $f$. Then for any
finite generating set $T$, there is a Floyd function $g$ such that
$\partial_{T, g} G$ is non-trivial.
\end{lem}
\begin{proof}
It is well-known that the identity map extends to a $(K,
0)$-quasi-isometric map $\varphi: \Gy \to \Gx$ for some constant $K
\ge 1$. Suppose that $\pGf$ is non-trivial, where the Floyd function
$f$ satisfies the $\lambda$-delay property for $\lambda \in ]0, 1[$.
We shall find a Floyd function $g$ such that $\varphi$ extends to a
continuous map $\hat \varphi:
\partial_{T, g} G \to \pGf$. This clearly shows the conclusion.

Recall a result of Gerasimov-Potyagailo \cite[Lemma 2.5]{GePo2} that
if $f(n)/g(Kn)$ is upper bounded by a constant, then $\varphi$
extends to a continuous map $\hat \varphi: \partial_{T, g} G \to
\pGf$. Without loss of generality, we assume that $K =2$. The
general case follows by a finite number of repeated applications of
the case $K=2$.

For $n\ge 1$, define $g(2n) = f(n)$ and $g(2n-1) = (f(n-1) +
f(n))/2$, where $g(0) = 0$ by convention. It is obvious that
$f(n)/g(2n) =1$. It suffices to verify that $g(n)$ is a Floyd
function. By direct computations one sees that $g$ satisfies the
$2\lambda/(1+\lambda)$-delay property. The proof is complete.
\end{proof}

In the sense of Lemma \ref{nontri} we can speak of the
non-triviality of Floyd boundary of a group. From now on, assume
that $G$ admits a nontrivial Floyd boundary.

In \cite{Ka}, Karlsson showed that the left multiplication on $G$
extends to a convergence group action of $G$ on $\pGf$. See
\cite{Bow2} for a reference on convergence groups actions.

\begin{lem} \label{hyperbolic}
Let $h$ be a hyperbolic element in $G$. Then $h$ is contracting.
\end{lem}
\begin{proof}
This is essentially proven in \cite{GePo4} as Proposition 8.2.4
therein with a greater generality: if a subgroup $H$ acts
cocompactly outside its limit set in $\Gf$, then $Ho$ is a
contracting subset.

It is a well-known fact in a convergence group that for a hyperbolic
element $h$, the subgroup $\langle h\rangle$ acts properly and
cocompactly on the complement in $\Gf$ of the two fixed points of
$h$. Hence, the conclusion follows.
\end{proof}
We below examine two particular cases of Floyd functions which have
appeared in the literatures.

In \cite{Floyd}, Floyd originally considered the case of "polynomial
type", for instance $f(n) =1/(n^2 +1)$. In this case, the shapes of
Floyd boundary behave nicely in the spirit of coarse geometry:  the
construction using different finite generating sets shall lead to
the homeomorphic Floyd boundaries.

The exponential type of Floyd function $f(n) = \lambda^n$ for
$\lambda \in ]0, 1[$ is exploited by Gerasimov-Potyagailo
\cite{Ge2}, \cite{GePo4},\cite{GePo2} and \cite{GePo3} in the study
of relatively hyperbolic groups. In \cite{Ge2}, Gerasimov showed
that for a non-elementary relatively hyperbolic group $G$ with a
finite generating set $S$, there exists $\lambda \in ]0, 1[$ such
that there is a continuous surjective map from the Floyd boundary
$\partial_{S, f} G$ for $f(n)=\lambda^n$ to the Bowditch boundary of
$G$. Hence, all non-elementary relatively hyperbolic groups have
non-trivial Floyd boundary. This proves Theorem \ref{tightfloyd}.

\section{Some open questions}
We conclude with some open questions about growth tightness. Define
the entropy $Ent(G)$ of a finite generated group $G$ to be the
infimum $\inf\{\g{G, S}\}$ over all finite generating sets. If there
exists $S$ such that $Ent(G)=\g {G, S}$, we say that entropy of $G$
is realized by $S$.
\begin{quest}\label{entropy}
Let $G$ be a relatively hyperbolic group. Is it true that $Ent(G)
>Ent(\bar G)$ for every quotient $\bar G=G/\Gamma$ where $\sharp
\Gamma=\infty$?
\end{quest}
\begin{rem}
It is an open problem that whether the entropy of a hyperbolic group
is realized or not by a finite generating set. By growth tightness
of hyperbolic groups, a positive solution would answer positively
Question \ref{entropy}. Note that there are many relatively
hyperbolic groups constructed in \cite{Sam2} (free products of a
non-Hopfian group with any non-trivial group) whose entropy are not
realized by any generating set.
\end{rem}

Taking account into Corollary \ref{tightCAT}, we do not know yet an
answer to the following.
\begin{quest}
Are CAT(0)-groups with rank-1 elements growth tight with respect to
word metrics?
\end{quest}

In the appendix, we show that a group with a contracting element
contains a hyperbolically embedded subgroup. In particular, it may
be natural to explore the growth tightness for the other interesting
classes of groups with a hyperbolically embedded subgroup.
\begin{quest}
Are mapping class groups growth tight? and what about $Out(F_n)$?
\end{quest}

Another interesting question is analogous to the spectrum gap
problem in \cite{DPPS} for critical exponents of Kleinian groups.
For simplicity, we only state it for word metrics. The version for a
group action on a metric space can be similarly formulated.
\begin{quest}\label{gap}
Suppose that the pair $(G, S)$ is growth tight. Let $\Delta(G,
S)=\inf \{\g {G, S}-\g{\bar G, \bar S}\}$ over all proper quotient
$\bar G$. Is it true that $\Delta(G, S)=0$?
\end{quest}
\begin{rem}
If $G$ is a relatively hyperbolic group, we answer positively the
question for word metrics, and hyperbolic metrics under mild
assumptions in \cite{YANG7}.
\end{rem}

\appendix
\setcounter{secnumdepth}{0}
\section{Appendix}

We include this appendix to make some connections with groups with a
hyperbolically embedded subgroup in \cite{DGO}.

\begin{lema}\label{hypembed}
Suppose $G$ acts properly on a geodesic metric space $(Y, d)$ with a
contracting subgroup $H$. Then $H$ is of finite index in a
hyperbolically embedded subgroup in $G$.
\end{lema}

\begin{rem}
An earlier result of Sisto in \cite{sisto} proved that a weakly
contracting element is always contained in a hyperbolically embedded
subgroup. Technically, our definition of contracting property
differs from his. It seems that a contracting element in the sense
of Sisto is contracting in this paper. Another difference is that
here this lemma handles with a slightly general situation where $H$
need not be generated by a contracting element.

Compared with Theorem 4.42 in \cite{DGO}, this lemma could produce
hyperbolically embedded subgroups which are not hyperbolic. For
example, every maximal parabolic subgroup in a relatively hyperbolic
group is contracting and is thus a hyperbolically embedded subgroup.
\end{rem}

By the discussion in Section 7, we note the following corollary
which seems not be recorded in the literature.
\begin{cora}\label{floydhyp}
If a group has a non-trivial Floyd boundary, then every hyperbolic
element is contained in a hyperbolically embedded subgroup.
\end{cora}

In the reminder of the appendix, we give a proof of Lemma
\ref{hypembed}. The outline follows the proof of Theorem 4.42 in
\cite{DGO}. Our arguments are greatly simplified and routine by
constructing appropriate admissible paths to use Proposition
\ref{admissible}. And our setting is more general, as $Y$ is not
assumed to be a hyperbolic space.

Let $\mathbb X$ be a contracting system with bounded intersection in
$Y$. Define $d_X(Z, W) =\diam{\proj_X(Z)\cup \proj_X(W)}$ for any
$X, Z, W \in \mathbb X$. We shall verify that $d_X$ has the
following properties.

\begin{lema}\label{projaxiom}
There exists $K>0$ such that
\begin{enumerate}
\item
at most one of $\{d_X(Z, W), d_Z(X, W), d_W(Z, X)\}$ is bigger than
$K$.
\item
$ \sharp \{X\in \mathbb X: d_X(Z, W)>K\} $ is finite for any two $Z,
W\in \mathbb X$.
\end{enumerate}
\end{lema}
\begin{proof}
Since $\mathbb X$ has $\nu$-bounded intersection, there exists $B>0$
such that $\diam{\proj_X(Z)} < B$ for two distinct $X, Z\in \mathbb
X$. By contracting property, the projection of any point in $Y$ to
$X\in \mathbb X$ is uniform upper bounded, say also by the constant
$B$.  Obviously for any $x \in \proj_X(Z), z\in \proj_Z(X)$, we have
$\diam{\proj_X([x, z])} <2B$. We denote by $p_{xz}$ an arbitrary
geodesic $[x, z]$ with $x \in \proj_X(Z), z\in \proj_Z(X)$.

Let $D=D(1, 0, \nu, 2B), R=R(1, 0, \nu, 2B)$ be given by Proposition
\ref{admissible}. Thus it suffices to set $K = D+4B$ for the
statement (1). Indeed, suppose to the contrary that $$d_X(Z, W),
d_W(X, Z) >K.$$ Then the path $$\gamma=p_{zx}\cdot [(p_{zx})_+,
(p_{xw})_-] \cdot p_{xw}\cdot [(p_{xw})_+, (p_{wz})_-]\cdot
p_{wz}\cdot [(p_{wz})_+, (p_{xz})_-]$$ is a $(D, 1, 0, \nu,
4B)$-admissible path. By Proposition \ref{admissible}, $\gamma$ is a
quasigeodesic. This is a contradiction, as $\gamma_-=\gamma_+$.

To prove (2), we fix a geodesic $p_{zw}$ between $Z, W$. Set
$K=D+4B+2R$. For any $X$ with $d_X(Z, W)>K$ we can construct an
admissible path with same endpoints as $p_{zw}$. By Proposition
\ref{admissible}, $p_{zw}$ is a $R$-fellow traveller. This implies
that $\diam{N_R(X) \cap p_{zw}} > D$. Clearly, there are only
finitely many such $X$, as different $N_R(X) \cap p_{zw}$ have
uniformly bounded overlap.
\end{proof}

\begin{proof}[Proof of Lemma \ref{hypembed}]
Denote $\mathbb X=\{gEo: g\in G\}$, where $E=E(H)$. By Lemma
\ref{projaxiom}, the set $\mathbb X$ with the distance function
$d_X(\cdot,\cdot)$ satisfies the axioms in \cite[Section 2.1]{BBF}.
Hence, we can construct a projection 1-complex $\mathcal C(\mathbb
X)$ which is a graph with vertex set $\mathbb X$. It is proved in
\cite{BBF} that $\mathcal C(\mathbb X)$ is a quasi-tree and thus is
a hyperbolic space.

Fix $X=Eo \in \mathbb X$. Let $Z=gEo\in \mathcal C(\mathbb X)$ be an
adjacent vertex to $X$. As in proof of Lemma \ref{projaxiom}, we
choose a geodesic $p_{xz}$ between $X, Z$ in $Y$. Then
$(p_{xz})_+=g_{xz}(p_{xz})_-$ for $g_{xz}\in G$. By construction, we
can choose $g_{xz} \in EgE$.

As $\mathcal C(\mathbb X)$ is connected, we see that
$S:=E\cup\{g_{xz}\}$ generates the group $G$. Indeed, for any $g\in
G$, we connect $X=Eo$ and $Z=gEo$ by a path in $\mathcal C(\mathbb
X)$. If $Z$ is adjacent to $X$ in $\mathcal C(\mathbb X)$, then
$EgE=Eg_{xz}E$ and thus $g \in Eg_{xz}E$. The general case follows,
as $g$ can be written as a finite product of elements in $Eg_{xz}E$.

We next prove that $\iota: X \in \mathbb X \to gE$ is a
quasi-isometric map between $(\mathcal C(\mathbb X), d_{\mathcal
C})$ and $(\Gx, \bar d)$, where $\bar d$ is the combinatorial metric
on $\Gx$. By the last paragraph, we have that $\bar d(1, g) \le
3d_{\mathcal C}(Eo, gEo)$.  On the other hand,  it is obvious that
$d_{\mathcal C}(Eo, gEo) \le d(1, g)$.

For any $g, h \in E$, we define a distance $\bar d_E(g, h)$ on $E$
to be the shortest distance of a path in $\Gx$ intersecting in $E$
only at its endpoints. For any $1, g \in E$, we will show that $d(o,
go) \le C \bar d_E(1, g)$ for some constant $C >0$. This clearly
concludes the proof, by definition of a hyperbolically embedded
subgroup in \cite{DGO}.

Let $p$ be a path between $1, g$ in $\Gx$. Write $g=x_1 x_2 \cdots
x_{n-1} x_n$, where $x_i\in E \cup \{g_{xz}\}$. For each generator
$x_i$ in $g$, we choose $p_i$ a geodesic between $x_1\cdots x_{i-1}
o, x_1\cdots x_i o$. In order to prove that $d(o, go) \le C \bar
d_E(1, g)$, it suffices to show that each $p_i$ has a uniform
bounded projection by $C$ to $X=Eo$.

By definition of $g_{xz}$,   $p_1$ and $p_n$ has bounded projection
to $Eo$. For other paths $p_i$, either $p_i$ has two endpoints in
$(p_i)_-Eo$, or $p_i$ is a geodesic between $(p_i)_-Eo$ and
$(p_i)_+Eo$. In the former case, $p_i$ has bounded projection to
$Eo$, as $(p_i)_-Eo$ does to $Eo$. In the latter case, $p_i$ also
has bounded projection to $Eo$. This follows from the construction
of a projection complex. In \cite{BBF}, it is proved that if two
vertices $Z, W$ are connected by an edge, then $\diam{d_X(Z \cup
W)}$ is uniformly small for any $X \in \mathbb X$. The proof is thus
complete.
\end{proof}


\ack The author is grateful to Anna Erschler for helpful
discussions. The author would like to thank Leonid Potyagailo for
his constant encouragement and support over years.

\bibliographystyle{amsplain}
\bibliography{bibliography}

\end{document}